\documentclass[12pt, reqno]{amsart}
\usepackage{amsmath, amsthm, amscd, amsfonts, amssymb, graphicx, color}
\usepackage[bookmarksnumbered, colorlinks, plainpages]{hyperref}
\hypersetup{colorlinks=true,linkcolor=red, anchorcolor=green, citecolor=cyan, urlcolor=red, filecolor=magenta, pdftoolbar=true}
\input{mathrsfs.sty}
\usepackage{xcolor}
\usepackage{xcolor}
\usepackage{ulem}

\newtheorem{theorem}{Theorem}[section]
\newtheorem{lemma}[theorem]{Lemma}

\newtheorem{cor}[theorem]{Corollary}
\theoremstyle{definition}

\theoremstyle{remark}
\newtheorem{remark}[theorem]{\bf{Remark}}
\numberwithin{equation}{section}
\begin{document}
		\title [Davis-Wielandt-Berezin radius inequalities]  { \Small{ Davis-Wielandt-Berezin radius inequalities of Reproducing kernel Hilbert space operators} }
	
	\author[ A. Sen, P. Bhunia, K. Paul] {Anirban Sen, Pintu Bhunia, Kallol Paul}

	\address [Sen] {Department of Mathematics, Jadavpur University, Kolkata 700032, West Bengal, India}
	\email{anirbansenfulia@gmail.com}
	
	\address [Bhunia] {Department of Mathematics, Jadavpur University, Kolkata 700032, West     Bengal, India}
	\email{pintubhunia5206@gmail.com}
	\email{pbhunia.math.rs@jadavpuruniversity.in}

	\address[Paul] {Department of Mathematics, Jadavpur University, Kolkata 700032, West Bengal, India}
	\email{kalloldada@gmail.com}
	\email{kallol.paul@jadavpuruniversity.in}
	
	\thanks{Mr. Anirban Sen would like to thank CSIR, Govt. of India for the financial support in the form of Junior Research Fellowship under the mentorship of Prof Kallol Paul. Mr. Pintu Bhunia sincerely acknowledges the financial support received from UGC, Govt. of India in the form of Senior Research Fellowship under the mentorship of Prof Kallol Paul.}
	
	
	
	\maketitle
	\begin{abstract}
Several upper and lower bounds of the Davis-Wielandt-Berezin radius of bounded linear operators defined on a reproducing kernel Hilbert space are given. Further, an inequality involving the Berezin number and the Davis-Wielandt-Berezin radius for the sum of two bounded linear operators is obtained, namely, if $A $ and $B$ are reproducing kernel Hilbert space operators, then
$$\eta(A+B) \leq \eta(A)+\eta(B)+\textbf{ber}(A^*B+B^*A),$$ where
$\eta(\cdot)$ and $\textbf{ber}(\cdot)$ are the Davis-Wielandt-Berezin radius and the Berezin number, respectively.
	\end{abstract}

	\section{\textbf{Introduction}}
	\noindent
	Let  $\mathbb{H}$ be a complex Hilbert space with the inner product $\langle.,.\rangle$ and $\|\cdot \|$ denotes the norm induced from the inner product. A bounded linear operator $A$ defined on $\mathbb{H}$ is said to be positive if $ \langle Ax, x \rangle \geq 0$ for each $ x \in \mathbb{H}.$ Let $\mathcal{L}(\mathbb{H})$ be the set of all bounded linear operators on  $\mathbb{H}.$ For  $A \in \mathcal{L}(\mathbb{H}),$ $A^*$ denotes the adjoint of $A$.    For  $A \in \mathcal{L}(\mathbb{H}),$ $|A|$ denotes the positive operator $(A^*A)^{1/2}.$ Every $A \in \mathcal{L}(\mathbb{H}) $ can be decomposed as,
	$A=\Re(A)+ { i }\Im(A),$
	where $\Re(A)=\frac{ A + A^*}{2} $ and $ \Im(A) = \frac{A - A^*}{2i}.$ This decomposition is called the Cartesian decomposition of $A.$ For any  $A \in \mathcal{L}(\mathbb{H}),$ its numerical range $W(A)$ is the subset $\big\{\langle Ax,x \rangle~:~ x \in \mathbb{H}, \|x\|=1\big\}$ of the complex plane  $\mathbb{C}.$ The numerical radius of $A$, denoted by $w(A)$, is defined as
	$ w(A) = \sup \left\{ |\lambda|~:~ \lambda \in W(A) \right\}.$ 
	Recall that the operator norm of $A \in \mathcal{L}(\mathbb{H})$ is defined by $\|A\|=\sup\left\{\|Ax\|~:~ x \in \mathbb{H}, \|x\|=1\right\}.$ 
	It is easy to verify that $w(.)$ defines a norm on  $\mathcal{L}(\mathbb{H}).$ Furthermore, it is equivalent to the operator norm on  $\mathcal{L}(\mathbb{H})$, satisfying 
	\begin{eqnarray*}
		\frac{1}{2}\|A\| \leq w(A) \leq \|A\|.
	\end{eqnarray*}
	For few recent non-trivial refinements of the above inequalities we refer the readers to see \cite{BBP20MIA,BP21ADM,BP21LAA,BP21BSM,BP21RIM,FK, FK2,TY}.  The importance of the study of the numerical range has led to its generalizations in different directions,
	one of them is the Davis-Wielandt shell. For any $A \in \mathcal{L}(\mathbb{H}),$ the Davis–Wielandt shell, denoted by $DW(A)$ and Davis–Wielandt
	radius, denoted by  $dw(A).$  are respectively defined as
	\begin{eqnarray*}
		DW(A)=\left\{\Big(\langle Ax,x \rangle,\|Ax\|^2\Big)~:~ x \in \mathbb{H}, \|x\|=1 \right\}
	\end{eqnarray*}
    and
	\begin{eqnarray*}
    	dw(A)=\sup\left\{\sqrt{|\langle Ax,x \rangle|^2+\|Ax\|^4}~:~ x \in \mathbb{H}, \|x\|=1 \right\}.
    \end{eqnarray*}
  
	Clearly, $dw(\cdot)$ is not a norm on  $\mathcal{L}(\mathbb{H})$ and it
	satisfies that the following inequality:
	\begin{eqnarray}\label{eqn0}
	\max{\{w(A),\|A\|^2\}} \leq dw(A) \leq \sqrt{w^2(A)+\|A\|^4}.
	\end{eqnarray}
In recent years  the Davis–Wielandt shell and the Davis–Wielandt radius inequalities have been studied by  many mathematicians, the readers can see \cite{BBP21BMMS,BBBP21AOFA,LP09GMJ,LPS08OM,ZS20MJM} and the references therein for recent results.

Next we turn our attention to a  reproducing kernel Hilbert space. A  reproducing kernel Hilbert space $\mathscr{H}=\mathscr{H}(\Omega)$ is a Hilbert space of all complex valued functions on a non-empty set $\Omega,$ which has the property that point evaluations are continuous, i.e., for every $\lambda \in \Omega$ the map $ E_{\lambda} : \mathscr{H} \to \mathbb{C}$  defined by $E_{\lambda}(f)=f(\lambda),$ is continuous. Throughout the article, a reproducing kernel Hilbert space on the set $\Omega$ is denoted by $\mathscr{H}.$ Riesz representation theorem ensures that for each $\lambda \in \Omega$ there exists a unique $k_{\lambda} \in \mathscr{H}$ such that $f(\lambda)=\langle f,k_{\lambda} \rangle,$ for all $f \in \mathscr{H}.$ The collection of $k_\lambda$ for all $\lambda \in \Omega$ is called  the reproducing kernel of $\mathscr{H}$ and the collection of $\hat{k}_{\lambda}=k_\lambda/\|k_\lambda\|$ for all $\lambda \in \Omega$ is called the normalized reproducing kernel of $\mathscr{H}.$ For any $A \in \mathcal{L}(\mathscr{H}),$ the Berezin symbol of $A$ is a function $\widetilde{A}$  on $\Omega$ defined as, $\widetilde{A}(\lambda)=\langle A\hat{k}_{\lambda},\hat{k}_{\lambda} \rangle,$  for each $\lambda \in \Omega,$  which was introduced by Berezin  \cite{BER1,BER2}. The Berezin set of $A$  is denoted by $\textbf{Ber}(A)$ and is defined as 
$		\textbf{Ber}(A)=\{\widetilde{A}(\lambda) : \lambda \in \Omega\}.$
	 The Berezin number (  see \cite{KAR06}) of $A,$ denoted by $ \textbf{ber}(A)  $ and the Berezin norm of $A,$  denoted by 	$\|A\|_{ber} ,$ are respectively  defined as 
	\begin{eqnarray*}
		\textbf{ber}(A) = \sup \left\{ \big| \widetilde{A}(\lambda) \big| : \lambda \in \Omega\right\}~~\mbox{and}~~
    	\|A\|_{ber}=\sup \left\{ \big|\langle A\hat{k}_{\lambda},\hat{k}_{\mu} \rangle\big| : \lambda, \mu \in \Omega\right\}.
    \end{eqnarray*}
    In \cite{BSP21ARXIV}, it is  proved that	
    $\left\|A\right\|_{ber}=  \textbf{ber}\left(A\right),$ if  $A\in \mathcal{L}(\mathscr{H})$ is positive. 
    Likewise as $\textbf{ber}(A),$ the least Berezin number of $A,$ denoted by $c(A)$, is defined as  $c(A)=\inf \left\{ \big| \widetilde{A}(\lambda) \big| : \lambda \in \Omega\right\}.$ It is clear from the definition that $ \textbf{Ber}(A) \subseteq W(A) $ and so $ \textbf{ber}(A) \leq w(A).$
    The Berezin number inequalities have been studied by many mathematicians over the years, for the latest and recent results we refer the readers to see \cite{b0,hlb,Bm,Trd} and the references therein. 
    
     Motivated by the Davis-Wielandt shell of a bounded linear operator on a Hilbert space, we here study the same (see \cite {dwb}) in the setting of a reproducing kernel Hilbert space. The Davis-Wielandt-Berezin shell of $A \in \mathcal{L}(\mathscr{H})$, denoted by $DW_{ber}(A)$, is defined as
    \begin{eqnarray*}
    	DW_{ber}(A)=\left\{\Big(\widetilde{A}(\lambda), \|A\hat{k}_{\lambda}\|^2 \Big)~:~\lambda \in \Omega \right\}.
    \end{eqnarray*}
 For $A \in \mathcal{L}(\mathscr{H}),$ the Davis-Wielandt-Berezin radius of $A$, denoted by $\eta(A)$,  is defined as
     \begin{eqnarray*}
      \eta(A)=\sup\left\{\sqrt{|\widetilde{A}(\lambda)|^2+\|A\hat{k}_{\lambda}\|^4}~:~\lambda \in \Omega \right\}.
    \end{eqnarray*}
	Likewise $dw(\cdot),$ it is easy to check that $\eta(\cdot)$ is not a norm on  $\mathcal{L}(\mathscr{H})$ and satisfies $\eta(A) \leq dw(A)$ for all $A \in \mathcal{L}(\mathscr{H}).$
	It is clear from the definition that for any  $A \in \mathcal{L}(\mathscr{H}),$ the Davis-Wielandt-Berezin radius satisfies the following inequality:
	\begin{eqnarray}\label{eqn1}
	\max\{ \textbf{ber}(A), \|A^*A\|_{ber}\} \leq \eta(A) \leq \sqrt{\textbf{ber}^2(A)+ \|A^*A\|^2_{ber}}.
	\end{eqnarray}
	
	In this paper,  we obtain upper and lower bounds for the Davis-Wielandt-Berezin radius of bounded linear operators on a reproducing kernel Hilbert space $\mathscr{H}.$
	We obtain an upper bound for the Davis-Wielandt-Berezin radius of the sum of two bounded linear operators. The bounds obtained here improve on the earlier ones studied in \cite{dwb}.

	\section{\textbf{Main Results}}
	
To reach our goal in this present article we begin with the following sequence of lemmas.

\begin{lemma}$($\cite{Simon}$)$.\label{lemma1}
		Let $A\in \mathcal{L}(\mathscr{H})$ be positive and $x\in \mathscr{H}$ with $\|x\|=1.$ Then  for all $r \geq 1,$
		\begin{eqnarray*}
		   \langle Ax,x\rangle^r\leq \langle A^rx,x\rangle.
		\end{eqnarray*}
	
	\end{lemma}
	
	\begin{lemma}$($\cite{F}$)$.\label{lemma2}
		Let $A\in \mathcal{L}(\mathscr{H})$ and $x, y\in \mathscr{H}$. Then for all $\alpha \in [0,1],$
		\[|\langle Ax,y\rangle|^2\leq \langle |A|^{2\alpha}x,x\rangle~~\langle |A^*|^{2(1-\alpha)}y,y\rangle.\] 
	\end{lemma}

     \begin{lemma}$($\cite{buzano}$)$.\label{lemma3}
      	Let $x,y,e\in \mathscr{H}$ with $\|e\|=1.$ Then 
    	\[|\langle x,e\rangle \langle e,y\rangle|\leq \frac{1}{2}\left(\|x\| \|y\|+|\langle x,y\rangle|\right).\]
    \end{lemma}

   \begin{lemma}$($\cite[P. 26]{hard}$)$.\label{lemma4}
  	For $a,b\geq 0,$ $0<\alpha<1,$ and $r\neq 0,$ let $M_{r}(a,b,\alpha)=\left(\alpha a^r+(1-\alpha)b^r \right)^{1/r}$ and $M_{0}(a,b,\alpha)=a^{\alpha}b^{1-\alpha}.$ Then for $r \leq s,$
  	$$M_{r}(a,b,\alpha)\leq M_{s}(a,b,\alpha).$$ 
  \end{lemma}

In the following theorem we present lower bounds for the Davis-Wielandt-Berezin radius of bounded linear operators.

\begin{theorem}\label{theo1}
	Let  $A \in \mathcal{L}(\mathscr{H}).$ Then 
	\begin{eqnarray*}
		(i)&& \eta^2(A) \geq  \max\{c^2(A)+\|A^*A\|^2_{ber}, \textbf{ber}^2(A)+c^2(A^*A) \},\\
		(ii)&& \eta^2(A) \geq 2 \max\{ \textbf{ber}(A)c(A^*A), c(A)\|A^*A\|_{ber} \},\\
		(iii)&& \eta^2(A) \geq \max\{c^2(A)(1+\|A^*A\|_{ber}),
		\textbf{ber}^2(A)(1+c(A^*A))  \}. 
	\end{eqnarray*}
\end{theorem}

\begin{proof}
	(i)	Let $\hat{k}_{\lambda}$ be a normalized reproducing kernel of $\mathscr{H}.$ Then from the definition of $\eta^2(A),$ we get
	   \begin{eqnarray*}
		   \eta^2(A) &\geq& |\langle A\hat{k}_{\lambda},\hat{k}_{\lambda} \rangle|^{2}+\|A\hat{k}_{\lambda}\|^4 \\
		   &=&|\langle A\hat{k}_{\lambda},\hat{k}_{\lambda} \rangle|^{2}+\langle A^*A\hat{k}_{\lambda},\hat{k}_{\lambda} \rangle^{2}\\
		   &\geq& c^2(A)+ \langle A^*A\hat{k}_{\lambda},\hat{k}_{\lambda} \rangle^{2}.
	   \end{eqnarray*}
    Taking supremum over all $\lambda\in\Omega$, we get
    \begin{eqnarray}\label{eqn2}
         \eta^2(A) \geq c^2(A)+\|A^*A\|^2_{ber}.
    \end{eqnarray}
    Again we have 
     \begin{eqnarray*}
    \eta^2(A) \geq |\langle A\hat{k}_{\lambda},\hat{k}_{\lambda} \rangle|^{2}+c^2(A^*A).
    \end{eqnarray*}
    Therefore, taking supremum over all $\lambda\in\Omega$, we get
    \begin{eqnarray}\label{eqn3}
    \eta^2(A) \geq  \textbf{ber}^2(A)+c^2(A^*A).
    \end{eqnarray}
    Combining (\ref{eqn2}) and  (\ref{eqn3}) we get the inequality (i).\\
    (ii) Let $\hat{k}_{\lambda}$ be a normalized reproducing kernel of $\mathscr{H}.$ Then we have, 
     \begin{eqnarray*}
    	 \eta^2(A) & \geq & |\langle A\hat{k}_{\lambda},\hat{k}_{\lambda} \rangle|^{2}+\langle A^*A\hat{k}_{\lambda},\hat{k}_{\lambda} \rangle^{2}\\ 
    	 & \geq & 2 |\langle A\hat{k}_{\lambda},\hat{k}_{\lambda} \rangle|\langle A^*A\hat{k}_{\lambda},\hat{k}_{\lambda} \rangle \\
    	  & \geq & 2  |\langle A\hat{k}_{\lambda},\hat{k}_{\lambda} \rangle|c(A^*A).
    \end{eqnarray*}
    Now,  taking supremum over all $\lambda\in\Omega$, we get
     \begin{eqnarray}\label{eqn4}
    \eta^2(A) \geq 2\textbf{ber}(A)c(A^*A) .
    \end{eqnarray}
    Also, we have
     \begin{eqnarray*}
    	\eta^2(A)  \geq 2c(A)\langle A^*A\hat{k}_{\lambda},\hat{k}_{\lambda} \rangle.
    \end{eqnarray*}
     Taking supremum over all $\lambda\in\Omega$, we get
      \begin{eqnarray}\label{eqn5}
     \eta^2(A) \geq 2  c(A)\|A^*A\|_{ber}.
     \end{eqnarray}
      Therefore, combining (\ref{eqn4}) and  (\ref{eqn5}) we get the inequality (ii).\\
      (iii)  Let $\hat{k}_{\lambda}$ be a normalized reproducing kernel of $\mathscr{H}.$  Then using Cauchy-Schwarz inequality we get, 
       \begin{eqnarray*}
      	\eta^2(A)  &\geq& |\langle A\hat{k}_{\lambda},\hat{k}_{\lambda} \rangle|^{2}+\|A\hat{k}_{\lambda}\|^4 \\
      	& \geq & |\langle A\hat{k}_{\lambda},\hat{k}_{\lambda} \rangle|^{2}+|\langle A\hat{k}_{\lambda},\hat{k}_{\lambda} \rangle|^{2}\|A\hat{k}_{\lambda}\|^2 \\
      	& = & |\langle A\hat{k}_{\lambda},\hat{k}_{\lambda} \rangle|^{2}(1+\langle A^*A\hat{k}_{\lambda},\hat{k}_{\lambda} \rangle) \\
      	& \geq & c^2(A)(1+\langle A^*A\hat{k}_{\lambda},\hat{k}_{\lambda} \rangle).
      \end{eqnarray*}
     Taking supremum over all $\lambda\in\Omega$, we get
      \begin{eqnarray}\label{eqn6}
          \eta^2(A) \geq c^2(A)(1+\|A^*A\|_{ber}).
       \end{eqnarray}
       Again, we have
         \begin{eqnarray*}
        	\eta^2(A)  \geq  |\langle A\hat{k}_{\lambda},\hat{k}_{\lambda} \rangle|^{2} (1+c(A^*A)).
        \end{eqnarray*}
     Now,  taking supremum over all $\lambda\in\Omega$, we get
    \begin{eqnarray}\label{eqn7}
    \eta^2(A) \geq \textbf{ber}^2(A)(1+c(A^*A)).
    \end{eqnarray}
     So, combining (\ref{eqn6}) and  (\ref{eqn7}) we get the inequality (iii).
\end{proof}

\begin{remark}
   Clearly,	the lower bound of $\eta(A)$ obtained in Theorem \ref{theo1} (i) improves the lower bound given in (\ref{eqn1}).
\end{remark}

Now, in the following Theorem we obtain an upper bound for the Davis-Wielandt-Berezin radius of bounded linear operators.

\begin{theorem}\label{theo2}
	Let  $A \in \mathcal{L}(\mathscr{H}).$ Then 
	\begin{eqnarray*}
	    \eta^2(A) \leq \min_{\theta \in [0,2\pi]}\textbf{ber}^2(|A|^2+e^{\rm i \theta}A)+2\|A^*A\|_{ber}\textbf{ber}(A).
	\end{eqnarray*}
\end{theorem}

\begin{proof}
	 Let $\hat{k}_{\lambda}$ be a normalized reproducing kernel of $\mathscr{H}.$ Then we get,
	  \begin{eqnarray*}
	 	 &&|\langle A\hat{k}_{\lambda},\hat{k}_{\lambda}\rangle|^{2} +\|A\hat{k}_{\lambda}\|^4\\
	 	 &=& |\langle A\hat{k}_{\lambda},\hat{k}_{\lambda}\rangle+\langle A\hat{k}_{\lambda},A \hat{k}_{\lambda}\rangle|^2-2\Re\left( \langle A\hat{k}_{\lambda},A \hat{k}_{\lambda}\rangle \langle A\hat{k}_{\lambda}, \hat{k}_{\lambda}\rangle\right)\\
	 	  & \leq & |\langle A\hat{k}_{\lambda},\hat{k}_{\lambda}\rangle+\langle |A|^2\hat{k}_{\lambda},\hat{k}_{\lambda}\rangle|^2+ 2 \langle A\hat{k}_{\lambda}, A\hat{k}_{\lambda}\rangle | \langle A\hat{k}_{\lambda}, \hat{k}_{\lambda}\rangle|\\
	 	 & = & |\langle (|A|^2+A)\hat{k}_{\lambda},\hat{k}_{\lambda}\rangle|^{2} + 2 \langle A^*A\hat{k}_{\lambda}, \hat{k}_{\lambda}\rangle | \langle A\hat{k}_{\lambda}, \hat{k}_{\lambda}\rangle|\\
	 	 & \leq & \textbf{ber}^2(|A|^2+A)+2\|A^*A\|_{ber}\textbf{ber}(A) \, \, \mbox{ (since $A^*A$ is positive).}
	 \end{eqnarray*}
  Taking supremum over all $\lambda\in\Omega$, we get
   \begin{eqnarray}\label{eqn8}
  \eta^2(A) \leq \textbf{ber}^2(|A|^2+A)+2\|A^*A\|_{ber}\textbf{ber}(A).
  \end{eqnarray}
  Now, replacing $A$ by $e^{\rm i \theta}A$ in (\ref{eqn8}) we get,
   \begin{eqnarray}\label{eqn9}
  \eta^2(A) \leq \textbf{ber}^2(|A|^2+e^{\rm i \theta}A)+2\|A^*A\|_{ber}\textbf{ber}(A).
  \end{eqnarray}
  Therefore, taking minimum over all $\theta \in [0,2\pi]$ in (\ref{eqn9}) we get the desired result.
\end{proof}

\begin{remark}
    Clearly, $\sup_{\lambda\in\Omega}\|A\hat{k}_{\lambda}\|^2=\|A^*A\|_{ber}.$ Therefore, the inequality obtained in \cite[Th. 1.]{dwb} is of the form
    \begin{eqnarray}\label{rmk1}
    \eta^2(A) \leq \textbf{ber}^2(|A|^2-A)+2\|A^*A\|_{ber}\textbf{ber}(A).
    \end{eqnarray}
    So, we remark that the upper bound of $ \eta(A)$ obtained in Theorem \ref{theo2} is sharper than the existing upper bound in (\ref{rmk1}). Considering $A=- I$ it easily follows that 
    \begin{eqnarray*}
         \textbf{ber}^2(|A|^2+e^{\rm i \theta}A)+2\|A^*A\|_{ber}\textbf{ber}(A)=2 \, \,\, (\text{for $\theta=0$})
     \end{eqnarray*}
      and 
     \begin{eqnarray*}
       \textbf{ber}^2(|A|^2-A)+2\|A^*A\|_{ber}\textbf{ber}(A)=6,
    \end{eqnarray*}
   so that 
    \begin{eqnarray*}
     \min_{\theta \in [0,2\pi]}\textbf{ber}^2(|A|^2+e^{\rm i \theta}A)+2\|A^*A\|_{ber}\textbf{ber}(A) < \textbf{ber}^2(|A|^2-A)+2\|A^*A\|_{ber}\textbf{ber}(A). 
    \end{eqnarray*}
\end{remark}

In the next Theorem we obtain an upper bound as well as a lower bound for the Davis-Wielandt-Berezin radius of bounded linear operators.

\begin{theorem}\label{theo3}
	Let  $A \in \mathcal{L}(\mathscr{H}).$ Then 
	\begin{eqnarray*}
	&&\frac{1}{2}\max\left\{\textbf{ber}^2(A+|A|^2)+c^2(A-|A|^2), \textbf{ber}^2(A-|A|^2)+ c^2(A+|A|^2) \right\}\\
	&&  \leq 	\eta^2(A)
	  \leq \frac{1}{2}\left\{\textbf{ber}^2(A+|A|^2)+\textbf{ber}^2(A-|A|^2)\right\}.
	\end{eqnarray*}
\end{theorem}

\begin{proof}
	Let $\hat{k}_{\lambda}$ be a normalized reproducing kernel of $\mathscr{H}.$ Then we have,

     \begin{eqnarray}\label{eqn11}
         	&&|\langle      A\hat{k}_{\lambda},\hat{k}_{\lambda}\rangle|^{2}+\|A\hat{k}_{\lambda}\|^4 \nonumber\\
         	&&=\frac{1}{2}|\langle A\hat{k}_{\lambda},\hat{k}_{\lambda}\rangle + \langle A\hat{k}_{\lambda},A\hat{k}_{\lambda}\rangle|^2+\frac{1}{2}|\langle A\hat{k}_{\lambda},\hat{k}_{\lambda}\rangle - \langle A\hat{k}_{\lambda},A\hat{k}_{\lambda}\rangle|^2\nonumber\\
         	&&=\frac{1}{2}|\langle A\hat{k}_{\lambda},\hat{k}_{\lambda}\rangle + \langle |A|^2\hat{k}_{\lambda},\hat{k}_{\lambda}\rangle|^2+\frac{1}{2}|\langle A\hat{k}_{\lambda},\hat{k}_{\lambda}\rangle - \langle |A|^2\hat{k}_{\lambda},\hat{k}_{\lambda}\rangle|^2\nonumber\\
         	&&=\frac{1}{2}|\langle (A+|A|^2)\hat{k}_{\lambda},\hat{k}_{\lambda}\rangle|^2+\frac{1}{2}|\langle (A-|A|^2)\hat{k}_{\lambda},\hat{k}_{\lambda}\rangle|^2.
     \end{eqnarray}
     Now, from (\ref{eqn11}), we get
     \begin{eqnarray*}
     	|\langle A\hat{k}_{\lambda},\hat{k}_{\lambda}\rangle|^{2}+\|A\hat{k}_{\lambda}\|^4 \leq  \frac{1}{2}\left\{\textbf{ber}^2(A+|A|^2)+\textbf{ber}^2(A-|A|^2)\right\}.
     \end{eqnarray*}
     Therefore, taking supremum over all $\lambda\in\Omega,$ we get
      \begin{eqnarray}\label{eqn13}
      \eta^2(A) \leq \frac{1}{2}\left\{\textbf{ber}^2(A+|A|^2)+\textbf{ber}^2(A-|A|^2)\right\}.
     \end{eqnarray}
    Again from the definition of $\eta(A)$ and using the equality (\ref{eqn11}), we get 
        \begin{eqnarray*}
       	  \frac{1}{2}\left\{|\langle (A+|A|^2)\hat{k}_{\lambda},\hat{k}_{\lambda}\rangle|^2+c^2(A-|A|^2)\right\} \leq |\langle      A\hat{k}_{\lambda},\hat{k}_{\lambda}\rangle|^{2}+\|A\hat{k}_{\lambda}\|^4.
       \end{eqnarray*}
   Taking supremum over all $\lambda\in\Omega,$ we get
   \begin{eqnarray}\label{eqn12a}
    \frac{1}{2}\textbf{ber}^2(A+|A|^2)+\frac{1}{2}c^2(A-|A|^2) \leq \eta^2(A).
   \end{eqnarray}
   Similarly,  from the equality (\ref{eqn11}), we get 
   \begin{eqnarray*}
   	 \frac{1}{2}\left\{c^2(A+|A|^2)+|\langle (A-|A|^2)\hat{k}_{\lambda},\hat{k}_{\lambda}\rangle|^2\right\} \leq |\langle      A\hat{k}_{\lambda},\hat{k}_{\lambda}\rangle|^{2}+\|A\hat{k}_{\lambda}\|^4.
   \end{eqnarray*}
   Taking supremum over all $\lambda\in\Omega,$ we get
   \begin{eqnarray}\label{eqn12b}
    \frac{1}{2}\textbf{ber}^2(A-|A|^2)+\frac{1}{2}c^2(A+|A|^2) \leq \eta^2(A).
   \end{eqnarray}
   Therefore, combining (\ref{eqn12a}) and (\ref{eqn12b}) we get the first inequality, and this completes the proof.
\end{proof}

Another upper bound of the Davis-Wielandt-Berezin radius of bounded linear operators reads as follows. 

\begin{theorem}\label{theo5}
	Let  $A \in \mathcal{L}(\mathscr{H}).$ Then for all $\alpha \in [0,1],$
	\begin{eqnarray*}
		\eta^2(A) \leq&& \frac{1}{4}\left\| |A|^{2\alpha}+|A^*|^{2(1-\alpha)}\right\|^2_{ber}\\
		&&+ \frac{1}{4}\textbf{ber}\left(2|A|^2+|A|^{2\alpha}-|A^*|^{2(1-\alpha)}\right)\textbf{ber}\left(2|A|^2-|A|^{2\alpha}+|A^*|^{2(1-\alpha)}\right).
	\end{eqnarray*}
    
\end{theorem}

\begin{proof}
	 Let $\hat{k}_{\lambda}$ be a normalized reproducing kernel of $\mathscr{H}.$ Then we have
	\begin{eqnarray*}
	 	&& |\langle      A\hat{k}_{\lambda},\hat{k}_{\lambda}\rangle|^{2}+\|A\hat{k}_{\lambda}\|^4 \\ 
	 	&& \leq \langle |A|^{2\alpha}\hat{k}_{\lambda},\hat{k}_{\lambda} \rangle \langle |A^*|^{2(1-\alpha)}\hat{k}_{\lambda},\hat{k}_{\lambda}\rangle+\langle      |A|^2\hat{k}_{\lambda},\hat{k}_{\lambda}\rangle^2\,\,\,\,\Big(\mbox{by Lemma \ref{lemma2}}\Big)\\
	 	&&= \frac{1}{4}\left(\langle (|A|^{2\alpha} + |A^*|^{2(1-\alpha)})\hat{k}_{\lambda},\hat{k}_{\lambda}\rangle^2-\langle (|A|^{2\alpha} - |A^*|^{2(1-\alpha)})\hat{k}_{\lambda},\hat{k}_{\lambda}\rangle^2\right)+\langle      |A|^2\hat{k}_{\lambda},\hat{k}_{\lambda}\rangle^2\\
	 	&& = \frac{1}{4} \langle (|A|^{2\alpha} + |A^*|^{2(1-\alpha)})\hat{k}_{\lambda},\hat{k}_{\lambda}\rangle^2+ \frac{1}{4}  \left(\langle 2|A|^2\hat{k}_{\lambda},\hat{k}_{\lambda}\rangle^2-\langle (|A|^{2\alpha} - |A^*|^{2(1-\alpha)})\hat{k}_{\lambda},\hat{k}_{\lambda}\rangle^2\right)\\
	 	&& = \frac{1}{4} \langle (|A|^{2\alpha} + |A^*|^{2(1-\alpha)})\hat{k}_{\lambda},\hat{k}_{\lambda}\rangle^2\\
	 	&&\,\,\,\,\,\,+\frac{1}{4}\langle  (2|A|^2+|A|^{2\alpha}-|A^*|^{2(1-\alpha)})\hat{k}_{\lambda},\hat{k}_{\lambda}\rangle \langle (2|A|^2-|A|^{2\alpha}+|A^*|^{2(1-\alpha)})\hat{k}_{\lambda},\hat{k}_{\lambda}\rangle\\
	 	&& \leq \frac{1}{4}\| |A|^{2\alpha}+|A^*|^{2(1-\alpha)}\|^2_{ber}\\
	 	&&\,\,\,\,\,\,+ \frac{1}{4}\textbf{ber}\left(2|A|^2+|A|^{2\alpha}-|A^*|^{2(1-\alpha)}\right)\textbf{ber}\left(2|A|^2-|A|^{2\alpha}+|A^*|^{2(1-\alpha)}\right).
	 \end{eqnarray*}
      Taking supremum over all $\lambda\in\Omega$, we get the desired result.
\end{proof}

In particular, choosing $\alpha =\frac{1}{2}$ in Theorem \ref{theo5} we get the following corollary.

\begin{cor}\label{cor1}
	Let  $A \in \mathcal{L}(\mathscr{H}).$ Then 
	\begin{eqnarray*}
		\eta^2(A) \leq \frac{1}{4}\left\| |A|+|A^*|\right\|^2_{ber}
		+ \frac{1}{4}\textbf{ber}\left(2|A|^2+|A|-|A^*|\right)\textbf{ber}\left(2|A|^2-|A|+|A^*|\right).
	\end{eqnarray*}
\end{cor}

Next upper bound of the Davis-Wielandt-Berezin radius of bounded linear operators reads as follows.

\begin{theorem}\label{theo6}
	Let  $A \in \mathcal{L}(\mathscr{H})$. Then 	for all $r \geq 1$ and  for all $\alpha \in [0,1],$ 
	\begin{eqnarray*}
		\eta^{4r}(A) \leq 2^{2r-2}\left\||A|^{4\alpha r}+|A|^{4r}\right\|_{ber} \left\||A^*|^{4(1-\alpha) r}+|A|^{4r}\right\|_{ber}.
	\end{eqnarray*}

\end{theorem}

\begin{proof}
	 Let $\hat{k}_{\lambda}$ be a normalized reproducing kernel of $\mathscr{H}.$ Then we have
	 	\begin{eqnarray*}
	 	   	&& |\langle      A\hat{k}_{\lambda},\hat{k}_{\lambda}\rangle|^{2}+\|A\hat{k}_{\lambda}\|^4 \\ 
	 	   && = 2\left(\frac{|\langle      A\hat{k}_{\lambda},\hat{k}_{\lambda}\rangle|^{2}+\|A\hat{k}_{\lambda}\|^4}{2}\right)\\
	 	   && \leq 2\left(\frac{|\langle      A\hat{k}_{\lambda},\hat{k}_{\lambda}\rangle|^{2r}+\|A\hat{k}_{\lambda}\|^{4r}}{2}\right)^{1/r}\,\,\,\,\Big(\mbox{by Lemma \ref{lemma4}}\Big)\\
	 	    && = 2^{1-1/r}\left(|\langle      A\hat{k}_{\lambda},\hat{k}_{\lambda}\rangle|^{2r}+\langle      |A|^2\hat{k}_{\lambda},\hat{k}_{\lambda}\rangle^{2r}\right)^{1/r}\\
	 	   && \leq  2^{1-1/r}\left(\langle |A|^{2\alpha}\hat{k}_{\lambda},\hat{k}_{\lambda} \rangle^r \langle |A^*|^{2(1-\alpha)}\hat{k}_{\lambda},\hat{k}_{\lambda}\rangle^r+\langle      |A|^2\hat{k}_{\lambda},\hat{k}_{\lambda}\rangle^{2r} \right)^{1/r}\,\,\,\,\Big(\mbox{by Lemma \ref{lemma2}}\Big)\\
	 	    && \leq  2^{1-1/r}\left(\langle |A|^{4\alpha r}\hat{k}_{\lambda},\hat{k}_{\lambda} \rangle^{1/2} \langle |A^*|^{4(1-\alpha)r}\hat{k}_{\lambda},\hat{k}_{\lambda}\rangle^{1/2}+\langle    |A|^{4r}\hat{k}_{\lambda},\hat{k}_{\lambda}\rangle^{1/2}\langle      |A|^{4r}\hat{k}_{\lambda},\hat{k}_{\lambda}\rangle^{1/2} \right)^{1/r}\\
	 	    &&\,\,\,\,\,\,\,\,\,\,\,\,\,\,\,\,\,\,\,\,\,\,\,\,\,\,\,\,\,\,\,\,\,\,\,\,\,\,\,\,\,\,\,\,\,\,\,\,\,\,\,\,\,\,\,\,\,\,\,\,\,\,\,\,\,\,\,\,\,\,\,\,\,\,\,\,\,\,\,\,\,\,\,\,\,\,\,\,\,\,\,\,\,\,\,\,\,\,\,\,\,\,\,\,\,\,\,\,\,\,\,\,\,\,\,\,\,\,\,\,\,\,\,\,\,\,\,\,\,\,\,\,\,\,\,\,\,\,\,\,\,\,\,\,\,\,\,\,\,\,\,\,\,\,\,\,\Big(\mbox{by Lemma \ref{lemma1}}\Big)\\
	 	    && \leq 2^{1-1/r}\left(\langle |A|^{4\alpha r}\hat{k}_{\lambda},\hat{k}_{\lambda} \rangle  +\langle    |A|^{4r}\hat{k}_{\lambda},\hat{k}_{\lambda}\rangle\right)^{1/2r}\left( \langle |A^*|^{4(1-\alpha)r}\hat{k}_{\lambda},\hat{k}_{\lambda}\rangle+ \langle      |A|^{4r}\hat{k}_{\lambda},\hat{k}_{\lambda}\rangle \right)^{1/2r}\\
	 	    && \,\,\,\,\,\,\,\, \,\,\,\,\,\,\,\,\Big(\mbox{since $(ab+cd)^2\leq(a^2+c^2)(b^2+d^2)$ for all $a,b,c,d \in \mathbb{R}$}\Big)\\
	 	    && = 2^{1-1/r}\langle (|A|^{4\alpha r}+|A|^{4r}) \hat{k}_{\lambda},\hat{k}_{\lambda} \rangle^{1/2r} \langle (|A^*|^{4(1-\alpha) r}+|A|^{4r}) \hat{k}_{\lambda},\hat{k}_{\lambda} \rangle^{1/2r}\\
	 	    &&  \leq 2^{1-1/r}\left\||A|^{4\alpha r}+|A|^{4r}\right\|^{1/2r}_{ber} \left\||A^*|^{4(1-\alpha) r}+|A|^{4r}\right\|^{1/2r}_{ber}.
	 	   \end{eqnarray*}
 	       So, taking supremum over all $\lambda\in\Omega$, we get the desired result.
\end{proof}

 In particular, considering $r=1$ in Theorem \ref{theo6} we get the following upper bound for the Davis-Wielandt-Berezin radius.
\begin{cor}\label{cor2}
	Let  $A \in \mathcal{L}(\mathscr{H}).$ Then 
	\begin{eqnarray}\label{eqn14}
		\eta^{4}(A) \leq \min_{ \alpha \in [0,1]} \left\||A|^{4\alpha }+|A|^{4}\right\|_{ber} \left\||A^*|^{4(1-\alpha) }+|A|^{4}\right\|_{ber},
	\end{eqnarray}

   In particular, for $\alpha=\frac{1}{2}$
   \begin{eqnarray}\label{eqn15}
   \eta^{4}(A) \leq \left\||A|^{2}+|A|^{4}\right\|_{ber} \left\||A^*|^{2}+|A|^{4}\right\|_{ber}.
   \end{eqnarray}
\end{cor}

\begin{remark}
	Let $A \in \mathcal{L}(\mathscr{H})$ be a normal operator. Then the inequality (\ref{eqn14}) is of the form
         \begin{eqnarray}\label{rmk2}
         \eta^{2}(A) \leq \min_{ \alpha \in [0,1]}\left\||A|^{4\alpha }+|A|^{4}\right\|^{1/2}_{ber} \left\||A|^{4(1-\alpha) }+|A|^{4}\right\|^{1/2}_{ber}.
        \end{eqnarray}
  Clearly, we have
     \begin{eqnarray*}
     	\min_{ \alpha \in [0,1]}\left\||A|^{4\alpha }+|A|^{4}\right\|^{1/2}_{ber} \left\||A|^{4(1-\alpha) }+|A|^{4}\right\|^{1/2}_{ber}
     	\leq \left\||A|^{2}+|A|^{4}\right\|_{ber}.
     \end{eqnarray*}
     Since $\|A\hat{k}_{\lambda}\|=\|A^*\hat{k}_{\lambda}\|$ for all $\lambda\in\Omega$,
      following \cite[Th. 2.]{dwb} we have,
       \begin{eqnarray*}
      	\eta^{2}(A)
      	\leq \left\||A|^{2}+|A|^{4}\right\|_{ber}.
      \end{eqnarray*}
    Therefore, we remark that the inequality (\ref{rmk2}) is better than the bound given in \cite[Th. 2]{dwb} for  normal operators. 
\end{remark}

Next upper bound reads as follows.

\begin{theorem}\label{theo7}
	Let  $A \in \mathcal{L}(\mathscr{H}).$ Then 
	\begin{eqnarray*}
		\eta^{2}(A) \leq  \min_{\alpha \in [0,1]}\left\{ \frac{1}{2}\left\|\left(|A|^{2\alpha}+|A^*|^{2(1-\alpha)}\right)^2+2|A|^4\right\|_{ber}-c\left(|A|^{2\alpha}\right)c\left(|A^*|^{2(1-\alpha)}\right)\right\}.
	\end{eqnarray*}
\end{theorem}

\begin{proof}
	    Let $\hat{k}_{\lambda}$ be a normalized reproducing kernel of $\mathscr{H}.$ Then
	    \begin{eqnarray*}
	    	&& |\langle      A\hat{k}_{\lambda},\hat{k}_{\lambda}\rangle|^{2}+\|A\hat{k}_{\lambda}\|^4 \\ 
	    	&& \leq \langle |A|^{2\alpha}\hat{k}_{\lambda},\hat{k}_{\lambda} \rangle \langle |A^*|^{2(1-\alpha)}\hat{k}_{\lambda},\hat{k}_{\lambda}\rangle+\langle      |A|^2\hat{k}_{\lambda},\hat{k}_{\lambda}\rangle^2\,\,\,\,\,\,\Big(\mbox{by Lemma \ref{lemma2}}\Big)\\
	    	&&  \leq \frac{1}{2}\left(\langle |A|^{2\alpha}\hat{k}_{\lambda},\hat{k}_{\lambda} \rangle^2 + \langle |A^*|^{2(1-\alpha)}\hat{k}_{\lambda},\hat{k}_{\lambda}\rangle^2 \right)+\langle |A|^2\hat{k}_{\lambda},\hat{k}_{\lambda}\rangle^2\\
	    	&&  = \frac{1}{2}\left(\langle |A|^{2\alpha}\hat{k}_{\lambda},\hat{k}_{\lambda} \rangle + \langle |A^*|^{2(1-\alpha)}\hat{k}_{\lambda},\hat{k}_{\lambda}\rangle \right)^2+\langle |A|^2\hat{k}_{\lambda},\hat{k}_{\lambda}\rangle^2\\
	    	&&\,\,\,\,\,\, -\langle |A|^{2\alpha}\hat{k}_{\lambda},\hat{k}_{\lambda} \rangle \langle |A^*|^{2(1-\alpha)}\hat{k}_{\lambda},\hat{k}_{\lambda}\rangle\\
	    	&&  = \frac{1}{2}\left\langle \left(|A|^{2\alpha} + |A^*|^{2(1-\alpha)}\right)\hat{k}_{\lambda},\hat{k}_{\lambda}\right \rangle ^2+\langle |A|^2\hat{k}_{\lambda},\hat{k}_{\lambda}\rangle^2
	    	 -\langle |A|^{2\alpha}\hat{k}_{\lambda},\hat{k}_{\lambda} \rangle \langle |A^*|^{2(1-\alpha)}\hat{k}_{\lambda},\hat{k}_{\lambda}\rangle\\
	    	 &&  \leq \frac{1}{2}\left\langle \left(|A|^{2\alpha} + |A^*|^{2(1-\alpha)}\right)^2\hat{k}_{\lambda},\hat{k}_{\lambda}\right \rangle +\langle |A|^4\hat{k}_{\lambda},\hat{k}_{\lambda}\rangle
	    	 -\langle |A|^{2\alpha}\hat{k}_{\lambda},\hat{k}_{\lambda} \rangle \langle |A^*|^{2(1-\alpha)}\hat{k}_{\lambda},\hat{k}_{\lambda}\rangle\\
	    	 && \,\,\,\, \,\,\,\,\,\,\,\,\,\,\,\,\,\,\,\,\,\,\,\,\,\,\,\,\,\,\,\,\,\,\,\,\,\,\,\,\,\,\,\,\,\,\,\,\,\,\,\,\,\,\,\,\,\,\,\,\,\,\,\,\,\,\,\,\,\,\,\,\,\,\,\,\,\,\,\,\,\,\,\,\,\,\,\,\,\,\,\,\,\,\,\,\,\,\,\,\,\,\,\,\,\,\,\,\,\,\,\,\,\,\,\,\,\,\,\,\,\,\,\,\,\,\,\,\,\,\,\,\,\,\,\,\,\,\,\,\,\,\,\,\,\,\,\,\,\,\,\,\,\,\,\,\,\Big(\mbox{by Lemma \ref{lemma1}}\Big)\\
	    	  &&  = \frac{1}{2}\left\langle \left(\left(|A|^{2\alpha} + |A^*|^{2(1-\alpha)}\right)^2+2|A|^4 \right)\hat{k}_{\lambda},\hat{k}_{\lambda}\right \rangle 
	    	 -\langle |A|^{2\alpha}\hat{k}_{\lambda},\hat{k}_{\lambda} \rangle \langle |A^*|^{2(1-\alpha)}\hat{k}_{\lambda},\hat{k}_{\lambda}\rangle\\
	    	 && \leq \frac{1}{2}\left\|\left(|A|^{2\alpha}+|A^*|^{2(1-\alpha)}\right)^2+2|A|^4\right\|_{ber}-c\left(|A|^{2\alpha}\right)c\left(|A^*|^{2(1-\alpha)}\right).
	    \end{eqnarray*}
    Now, taking supremum over all $\lambda\in\Omega$, we get
    \begin{eqnarray}\label{eqn16}
    	\eta^{2}(A) \leq \frac{1}{2}\left\|\left(|A|^{2\alpha}+|A^*|^{2(1-\alpha)}\right)^2+2|A|^4\right\|_{ber}-c\left(|A|^{2\alpha}\right)c\left(|A^*|^{2(1-\alpha)}\right).
    \end{eqnarray}
    Clearly, the inequality (\ref{eqn16}) holds for all $\alpha \in [0,1].$ Therefore, taking minimum over all $\alpha \in [0,1],$ we get the required result.
 \end{proof}

In particular, choosing $\alpha = \frac{1}{2}$ in Theorem \ref{theo7} we get the following inequality.

\begin{cor}\label{cor3}
	Let  $A \in \mathcal{L}(\mathscr{H}).$ Then 
		\begin{eqnarray*}
		\eta^{2}(A) \leq   \left\|\left(|A|+|A^*|\right)^2+2|A|^4\right\|_{ber}
		-c\left(|A|\right)c\left(|A^*|\right).
	\end{eqnarray*}
\end{cor}


Our next bound reads as:

\begin{theorem}\label{theo8}
	Let  $A \in \mathcal{L}(\mathscr{H}).$ Then 
	\begin{eqnarray*}
		\eta^{2}(A) \leq   \min_{\alpha \in [0,1]}\left\| \alpha |A|^2+(1-\alpha)|A^*|^2+|A|^4\right\|_{ber}.
	\end{eqnarray*}
\end{theorem}

\begin{proof}
	 Let $\hat{k}_{\lambda}$ be a normalized reproducing kernel of $\mathscr{H}.$ Then for any $\alpha \in [0,1]$, by using Cauchy-Schwarz inequality, we get
	 \begin{eqnarray*}
	       |\langle A\hat{k}_{\lambda},\hat{k}_{\lambda}\rangle|^{2} && =\alpha |\langle A\hat{k}_{\lambda},\hat{k}_{\lambda}\rangle|^{2} + (1-\alpha)|\langle A\hat{k}_{\lambda},\hat{k}_{\lambda}\rangle|^{2}\\
	      &&  =\alpha |\langle A\hat{k}_{\lambda},\hat{k}_{\lambda}\rangle|^{2} + (1-\alpha)|\langle \hat{k}_{\lambda},A^*\hat{k}_{\lambda}\rangle|^{2}\\
	      && \leq \alpha \|A\hat{k}_{\lambda}\|^2+(1-\alpha) \|A^*\hat{k}_{\lambda}\|^2\\
	      &&  = \alpha \langle |A|^2\hat{k}_{\lambda},\hat{k}_{\lambda}\rangle+(1-\alpha)\langle |A^*|^2\hat{k}_{\lambda},\hat{k}_{\lambda}\rangle\\
	      && = \left\langle (\alpha |A|^2+(1-\alpha)|A^*|^2)\hat{k}_{\lambda},\hat{k}_{\lambda}\right\rangle. 
	 \end{eqnarray*}
     Therefore, we have
	  \begin{eqnarray*}
	 	&& |\langle A\hat{k}_{\lambda},\hat{k}_{\lambda}\rangle|^{2}+\|A\hat{k}_{\lambda}\|^4\\
	 	&& \leq \left\langle (\alpha |A|^2+(1-\alpha)|A^*|^2)\hat{k}_{\lambda},\hat{k}_{\lambda}\right\rangle + \langle |A|^2\hat{k}_{\lambda},\hat{k}_{\lambda}\rangle ^{2}\\
	 	&& \leq \left\langle (\alpha |A|^2+(1-\alpha)|A^*|^2)\hat{k}_{\lambda},\hat{k}_{\lambda}\right\rangle + \langle |A|^4\hat{k}_{\lambda},\hat{k}_{\lambda}\rangle \,\,\,\,\,\,\Big(\mbox{by Lemma \ref{lemma1}}\Big)\\
	 	&& =\left\langle (\alpha |A|^2+(1-\alpha)|A^*|^2+|A|^4)\hat{k}_{\lambda},\hat{k}_{\lambda}\right\rangle\\
	 	&& \leq \left\| \alpha |A|^2+(1-\alpha)|A^*|^2+|A|^4\right\|_{ber}.
	 \end{eqnarray*}
  Taking supremum over all $\lambda\in\Omega$, we get
   \begin{eqnarray}\label{eqn17}
  \eta^{2}(A) \leq \left\| \alpha |A|^2+(1-\alpha)|A^*|^2+|A|^4\right\|_{ber}.
  \end{eqnarray}
   Clearly for all $\alpha \in [0,1],$ the inequality (\ref{eqn17}) holds. Therefore, taking minimum over all $\alpha \in [0,1],$ we get the desired result.
\end{proof}

\begin{theorem}\label{theo9}
	Let  $A \in \mathcal{L}(\mathscr{H}).$ Then 
	\begin{eqnarray*}
		\eta^{2}(A) \leq   \min\{ \beta_1(A),\beta_2(A)\},
	\end{eqnarray*}
where
   \begin{eqnarray*}
   	 \beta_1(A)= \min_{ \alpha \in [0,1]} \left\{\frac{\alpha}{2}\textbf{ber}(A^2)  +\left\| \frac{\alpha}{4}|A|^{2}+\left(1-\frac{3}{4}\alpha\right)|A^*|^{2} +|A|^4 \right\|_{ber}\right\}
   \end{eqnarray*}
and 
   \begin{eqnarray*}
  	\beta_2(A) =  \min_{ \alpha \in [0,1]} \left\{\frac{\alpha}{2}\textbf{ber}(A^2)  +\left\| \left(1-\frac{3}{4}\alpha\right)|A|^{2}+\frac{\alpha}{4}|A^*|^{2} +|A|^4 \right\|_{ber}\right\}.
  \end{eqnarray*}
\end{theorem}

\begin{proof}
	 Let $\hat{k}_{\lambda}$ be a normalized reproducing kernel of $\mathscr{H}$ and $\alpha \in [0,1].$ Then using Cauchy-Schwarz inequality, we get
	  \begin{eqnarray*}
	       |\langle A\hat{k}_{\lambda},\hat{k}_{\lambda}\rangle|^{2} && \leq   \alpha |\langle A\hat{k}_{\lambda},\hat{k}_{\lambda}\rangle|^{2} + (1-\alpha) \|A^*\hat{k}_{\lambda}\|^2\\
	       && = \alpha |\langle A\hat{k}_{\lambda},\hat{k}_{\lambda}\rangle|^{2} + (1-\alpha) \langle |A^*|^2\hat{k}_{\lambda},\hat{k}_{\lambda}\rangle.
	 \end{eqnarray*}
    So, we get 
     \begin{eqnarray*}
    	&& |\langle A\hat{k}_{\lambda},\hat{k}_{\lambda}\rangle|^{2}+\|A\hat{k}_{\lambda}\|^4\\
    	&& \leq \alpha |\langle A\hat{k}_{\lambda},\hat{k}_{\lambda}\rangle|^{2} + (1-\alpha) \langle |A^*|^2\hat{k}_{\lambda},\hat{k}_{\lambda}\rangle+\langle |A|^2\hat{k}_{\lambda},\hat{k}_{\lambda}\rangle^2\\
    	&& \leq \frac{\alpha}{2}\left(|\langle A\hat{k}_{\lambda},A^*\hat{k}_{\lambda}\rangle|+\|A\hat{k}_{\lambda}\|\|A^*\hat{k}_{\lambda}\|\right) + (1-\alpha) \langle |A^*|^2\hat{k}_{\lambda},\hat{k}_{\lambda}\rangle+\langle |A|^2\hat{k}_{\lambda},\hat{k}_{\lambda}\rangle^2\\
    	&&\,\,\,\,\,\,\,\,\,\,\,\,\,\,\,\,\,\,\,\,\,\,\,\,\,\,\,\,\,\,\,\,\,\,\,\,\,\,\,\,\,\,\,\,\,\,\,\,\,\,\,\,\,\,\,\,\,\,\,\,\,\,\,\,\,\,\,\,\,\,\,\,\,\,\,\,\,\,\,\,\,\,\,\,\,\,\,\,\,\,\,\,\,\,\,\,\,\,\,\,\,\,\,\,\,\,\,\,\,\,\,\,\,\,\,\,\,\,\,\,\,\,\,\,\,\,\,\,\,\,\,\,\,\,\,\,\,\,\,\,\,\,\,\Big(\mbox{by Lemma \ref{lemma3}}\Big)\\
    	&& \leq \frac{\alpha}{2}|\langle A^2\hat{k}_{\lambda},\hat{k}_{\lambda}\rangle|+\frac{\alpha}{4}\left(\|A\hat{k}_{\lambda}\|^2+\|A^*\hat{k}_{\lambda}\|^2\right) + (1-\alpha) \langle |A^*|^2\hat{k}_{\lambda},\hat{k}_{\lambda}\rangle+\langle |A|^2\hat{k}_{\lambda},\hat{k}_{\lambda}\rangle^2\\
    	&& = \frac{\alpha}{2}|\langle A^2\hat{k}_{\lambda},\hat{k}_{\lambda}\rangle|+ \left\langle \left(\frac{\alpha}{4}|A|^{2}+\left(1-\frac{3}{4}\alpha\right)|A^*|^{2}\right)\hat{k}_{\lambda},\hat{k}_{\lambda}\right\rangle+\langle |A|^2\hat{k}_{\lambda},\hat{k}_{\lambda}\rangle^2\\
    	&& \leq \frac{\alpha}{2}|\langle A^2\hat{k}_{\lambda},\hat{k}_{\lambda}\rangle|+ \left\langle \left(\frac{\alpha}{4}|A|^{2}+\left(1-\frac{3}{4}\alpha\right)|A^*|^{2}\right)\hat{k}_{\lambda},\hat{k}_{\lambda}\right\rangle+\langle |A|^4\hat{k}_{\lambda},\hat{k}_{\lambda}\rangle\\
    	&&\,\,\,\,\,\,\,\,\,\,\,\,\,\,\,\,\,\,\,\,\,\,\,\,\,\,\,\,\,\,\,\,\,\,\,\,\,\,\,\,\,\,\,\,\,\,\,\,\,\,\,\,\,\,\,\,\,\,\,\,\,\,\,\,\,\,\,\,\,\,\,\,\,\,\,\,\,\,\,\,\,\,\,\,\,\,\,\,\,\,\,\,\,\,\,\,\,\,\,\,\,\,\,\,\,\,\,\,\,\,\,\,\,\,\,\,\,\,\,\,\,\,\,\,\,\,\,\,\,\,\,\,\,\,\,\,\,\,\,\,\,\,\,\Big(\mbox{by Lemma \ref{lemma1}}\Big)\\
    	&& = \frac{\alpha}{2}|\langle A^2\hat{k}_{\lambda},\hat{k}_{\lambda}\rangle|+ \left\langle \left(\frac{\alpha}{4}|A|^{2}+\left(1-\frac{3}{4}\alpha\right)|A^*|^{2}+|A|^4\right)\hat{k}_{\lambda},\hat{k}_{\lambda}\right\rangle\\
    	&& \leq \frac{\alpha}{2}\textbf{ber}(A^2)  +\left\| \frac{\alpha}{4}|A|^{2}+\left(1-\frac{3}{4}\alpha\right)|A^*|^{2} +|A|^4 \right\|_{ber}.
    \end{eqnarray*}
     Therefore, taking supremum over all $\lambda\in\Omega$, we get
      \begin{eqnarray}\label{eqn18}
     	\eta^2(A) \leq \frac{\alpha}{2}\textbf{ber}(A^2)  +\left\| \frac{\alpha}{4}|A|^{2}+\left(1-\frac{3}{4}\alpha\right)|A^*|^{2} +|A|^4 \right\|_{ber}.
     \end{eqnarray}
     As the inequality holds for all $\alpha \in [0,1],$ taking minimum over all $\alpha \in [0,1]$ we get
      \begin{eqnarray}\label{eqn19}
     \eta^2(A) \leq  \min_{ \alpha \in [0,1]} \left\{\frac{\alpha}{2}\textbf{ber}(A^2)  +\left\| \frac{\alpha}{4}|A|^{2}+\left(1-\frac{3}{4}\alpha\right)|A^*|^{2} +|A|^4 \right\|_{ber}\right\}.
     \end{eqnarray}
      Again using Cauchy-Schwarz inequality, we also get
      \begin{eqnarray*}
      	|\langle A\hat{k}_{\lambda},\hat{k}_{\lambda}\rangle|^{2} && \leq   \alpha |\langle A\hat{k}_{\lambda},\hat{k}_{\lambda}\rangle|^{2} + (1-\alpha) \|A\hat{k}_{\lambda}\|^2\\
      	&& = \alpha |\langle A\hat{k}_{\lambda},\hat{k}_{\lambda}\rangle|^{2} + (1-\alpha) \langle |A|^2\hat{k}_{\lambda},\hat{k}_{\lambda}\rangle.
      \end{eqnarray*}
  Proceeding similarly as above, we also get
   \begin{eqnarray}\label{eqn20}
  \eta^2(A) \leq  \min_{ \alpha \in [0,1]} \left\{\frac{\alpha}{2}\textbf{ber}(A^2)  +\left\| \left(1-\frac{3}{4}\alpha\right)|A|^{2}+\frac{\alpha}{4}|A^*|^{2} +|A|^4 \right\|_{ber}\right\}.
  \end{eqnarray}
  Therefore, combining (\ref{eqn19}) and (\ref{eqn20}) we get the desired result.
\end{proof}

In particular, choosing $\alpha=0$ and $\alpha=1$, respectively, in Theorem \ref{theo9} we get the following corollary.

\begin{cor}\label{cor4}
	Let  $A \in \mathcal{L}(\mathscr{H}).$ Then 
	\begin{eqnarray*}
		(i)&& \eta^2(A) \leq \min\left\{ \left\| |A|^2+|A|^4 \right\|_{ber},\left\||A^*|^2+|A|^4 \right\|_{ber} \right\},\\
		(ii)&& \eta^2(A) \leq \frac{1}{4}\left\| |A|^2+|A^*|^2 +4|A|^4\right\|_{ber}+\frac{1}{2}\textbf{ber}(A^2).
	\end{eqnarray*}
\end{cor}

Next upper bound for the Davis-Wielandt-Berezin radius is as follows.

\begin{theorem}\label{theo10}
	Let  $A \in \mathcal{L}(\mathscr{H}).$ Then 
	\begin{eqnarray*}
		\eta^{2}(A) \leq   \min\{ \gamma_1(A),\gamma_2(A)\},
	\end{eqnarray*}
	where
	\begin{eqnarray*}
		\gamma_1 (A) \min_{ \alpha \in [0,1]} \left\| \alpha \left(\frac{|A|+|A^*|}{2}\right)^2+(1-\alpha)|A|^2+|A|^4 \right\|_{ber}
	\end{eqnarray*}
	and 
	\begin{eqnarray*}
		\gamma_2(A) =  \min_{ \alpha \in [0,1]} \left\| \alpha \left(\frac{|A|+|A^*|}{2}\right)^2+(1-\alpha)|A^*|^2+|A|^4 \right\|_{ber}.
	\end{eqnarray*}
\end{theorem}

\begin{proof}
	 Let $\hat{k}_{\lambda}$ be a normalized reproducing kernel of $\mathscr{H}$ and $\alpha \in [0,1].$ Then using Cauchy-Schwarz inequality, we get
	\begin{eqnarray*}
		|\langle A\hat{k}_{\lambda},\hat{k}_{\lambda}\rangle|^{2} && \leq   \alpha |\langle A\hat{k}_{\lambda},\hat{k}_{\lambda}\rangle|^{2} + (1-\alpha) \|A\hat{k}_{\lambda}\|^2\\
		&& = \alpha |\langle A\hat{k}_{\lambda},\hat{k}_{\lambda}\rangle|^{2} + (1-\alpha) \langle |A|^2\hat{k}_{\lambda},\hat{k}_{\lambda}\rangle.
	\end{eqnarray*}
    Therefore, we get
    \begin{eqnarray*}
         	&& |\langle A\hat{k}_{\lambda},\hat{k}_{\lambda}\rangle|^{2}+\|A\hat{k}_{\lambda}\|^4\\
         	&& \leq  \alpha |\langle A\hat{k}_{\lambda},\hat{k}_{\lambda}\rangle|^{2} + (1-\alpha) \langle |A|^2\hat{k}_{\lambda},\hat{k}_{\lambda}\rangle+ \langle |A|^2\hat{k}_{\lambda},\hat{k}_{\lambda}\rangle^2\\
         	&& \leq \alpha \left( \langle |A|\hat{k}_{\lambda},\hat{k}_{\lambda}\rangle^{1/2}\langle |A^*|\hat{k}_{\lambda},\hat{k}_{\lambda}\rangle^{1/2}\right)^2 + (1-\alpha) \langle |A|^2\hat{k}_{\lambda},\hat{k}_{\lambda}\rangle+ \langle |A|^2\hat{k}_{\lambda},\hat{k}_{\lambda}\rangle^2\\
         	&&\,\,\,\,\,\,\,\,\,\,\,\,\,\,\,\,\,\,\,\,\,\,\,\,\,\,\,\,\,\,\,\,\,\,\,\,\,\,\,\,\,\,\,\,\,\,\,\,\,\,\,\,\,\,\,\,\,\,\,\,\,\,\,\,\,\,\,\,\,\,\,\,\,\,\,\,\,\,\,\,\,\,\,\,\,\,\,\,\,\,\,\,\,\,\,\,\,\,\,\,\,\,\,\,\,\,\,\,\,\,\,\,\,\,\,\,\,\,\,\,\,\,\,\,\,\,\,\,\,\,\,\,\,\,\,\,\,\,\,\Big(\mbox{by Lemma \ref{lemma2}}\Big)\\
         	&& \leq \alpha \left( \frac{\langle |A|\hat{k}_{\lambda},\hat{k}_{\lambda}\rangle + \langle |A^*|\hat{k}_{\lambda},\hat{k}_{\lambda}\rangle}{2}\right)^2 + (1-\alpha) \langle |A|^2\hat{k}_{\lambda},\hat{k}_{\lambda}\rangle+ \langle |A|^2\hat{k}_{\lambda},\hat{k}_{\lambda}\rangle^2\\
         	&& = \alpha  \left\langle \left(\frac{|A|+|A^*|}{2}\right)\hat{k}_{\lambda},\hat{k}_{\lambda}\right\rangle^2 + (1-\alpha) \langle |A|^2\hat{k}_{\lambda},\hat{k}_{\lambda}\rangle+ \langle |A|^2\hat{k}_{\lambda},\hat{k}_{\lambda}\rangle^2\\
         	&& \leq \alpha  \left\langle \left(\frac{|A|+|A^*|}{2}\right)^2\hat{k}_{\lambda},\hat{k}_{\lambda}\right\rangle + (1-\alpha) \langle |A|^2\hat{k}_{\lambda},\hat{k}_{\lambda}\rangle+ \langle |A|^4\hat{k}_{\lambda},\hat{k}_{\lambda}\rangle\\
         	&&\,\,\,\,\,\,\,\,\,\,\,\,\,\,\,\,\,\,\,\,\,\,\,\,\,\,\,\,\,\,\,\,\,\,\,\,\,\,\,\,\,\,\,\,\,\,\,\,\,\,\,\,\,\,\,\,\,\,\,\,\,\,\,\,\,\,\,\,\,\,\,\,\,\,\,\,\,\,\,\,\,\,\,\,\,\,\,\,\,\,\,\,\,\,\,\,\,\,\,\,\,\,\,\,\,\,\,\,\,\,\,\,\,\,\,\,\,\,\,\,\,\,\,\,\,\,\,\,\,\,\,\,\,\,\,\,\,\,\,\Big(\mbox{by Lemma \ref{lemma1}}\Big)\\
         	&& =  \left\langle \left(\alpha \left(\frac{|A|+|A^*|}{2}\right)^2+(1-\alpha)|A|^2+|A|^4\right)\hat{k}_{\lambda},\hat{k}_{\lambda}\right\rangle \\
         	&& \leq   \left\| \alpha \left(\frac{|A|+|A^*|}{2}\right)^2+(1-\alpha)|A|^2+|A|^4 \right\|_{ber}.
    \end{eqnarray*}
     Therefore, taking supremum over all $\lambda\in\Omega$, we get
      \begin{eqnarray}\label{eqn21}
     \eta^2(A) \leq  \left\| \alpha \left(\frac{|A|+|A^*|}{2}\right)^2+(1-\alpha)|A|^2+|A|^4 \right\|_{ber}.
     \end{eqnarray}
      As the inequality holds for all $\alpha \in [0,1],$ taking minimum over all $\alpha \in [0,1]$ we get
     \begin{eqnarray}\label{eqn22}
     \eta^2(A) \leq  \min_{ \alpha \in [0,1]} \left\| \alpha \left(\frac{|A|+|A^*|}{2}\right)^2+(1-\alpha)|A|^2+|A|^4 \right\|_{ber}.
     \end{eqnarray}
      Again using Cauchy-Schwarz inequality, we also get
     \begin{eqnarray*}
     	|\langle A\hat{k}_{\lambda},\hat{k}_{\lambda}\rangle|^{2} && \leq   \alpha |\langle A\hat{k}_{\lambda},\hat{k}_{\lambda}\rangle|^{2} + (1-\alpha) \|A^*\hat{k}_{\lambda}\|^2\\
     	&& = \alpha |\langle A\hat{k}_{\lambda},\hat{k}_{\lambda}\rangle|^{2} + (1-\alpha) \langle |A^*|^2\hat{k}_{\lambda},\hat{k}_{\lambda}\rangle.
     \end{eqnarray*}
     Proceeding similarly as above, we also get
     \begin{eqnarray}\label{eqn23}
     \eta^2(A) \leq  \min_{ \alpha \in [0,1]}\left\| \alpha \left(\frac{|A|+|A^*|}{2}\right)^2+(1-\alpha)|A^*|^2+|A|^4 \right\|_{ber}. 
     \end{eqnarray}
     Therefore, combining (\ref{eqn22}) and (\ref{eqn23}) we get the desired result.
\end{proof}

Finally, we obtain an upper bound for the Davis–Wielandt-Berezin radius of the sum of two bounded linear operators.

\begin{theorem}\label{th}
	Let  $A,B \in \mathcal{L}(\mathscr{H}).$ Then
	\begin{eqnarray*}
		\eta(A+B) \leq \eta(A)+\eta(B)+\textbf{ber}(A^*B+B^*A).
	\end{eqnarray*}
\end{theorem}

\begin{proof}
	Let  $\hat{k}_{\lambda}$ be the normalized reproducing kernel of $\mathscr{H}.$ Then using the definition of Davis-Wielandt-Berezin shell, we have
	\begin{eqnarray*}
		&& DW_{ber}(A+B)\\
		&& =\left\{\Big(\widetilde{A+B}(\lambda), \|(A+B)\hat{k}_{\lambda}\|^2\Big)~:~\lambda \in \Omega \right\}\\
		&& =\left\{\Big(\widetilde{A}(\lambda)+\widetilde{B}(\lambda), \langle(A+B)\hat{k}_{\lambda}, (A+B)\hat{k}_{\lambda}\rangle\Big)~:~\lambda \in \Omega\right\}\\
		&& =\left\{\Big(\widetilde{A}(\lambda)+\widetilde{B}(\lambda), \|A\hat{k}_{\lambda}\|^2+\|B\hat{k}_{\lambda}\|^2+ \langle(A^*B+B^*A)\hat{k}_{\lambda},\hat{k}_{\lambda}\rangle\Big)~:~\lambda \in \Omega\right\}\\
		&& =\left\{\Big(\widetilde{A}(\lambda), \|A\hat{k}_{\lambda}\|^2\Big)+\Big(\widetilde{B}(\lambda), \|B\hat{k}_{\lambda}\|^2\Big)+\Big(0, \langle(A^*B+B^*A)\hat{k}_{\lambda},\hat{k}_{\lambda}\rangle\Big)~:~\lambda \in \Omega \right\}\\
		&& \subseteq DW_{ber}(A)+ DW_{ber}(B)+ U,
	\end{eqnarray*}
	where $U=\left\{\Big(0, \langle(A^*B+B^*A)\hat{k}_{\lambda},\hat{k}_{\lambda}\rangle\Big)~:~\lambda \in \Omega \right\}.$
	This implies the desired inequality.
\end{proof}

\begin{remark}

It is easy to observe that $\eta(A+B) \leq \eta(A)+\eta(B),$ if
$Re \langle A\hat{k}_{\lambda},B\hat{k}_{\lambda} \rangle=0$  for all normalized reproducing kernel $\hat{k}_{\lambda}$ of $\mathscr{H}.$ 

\end{remark}

\bibliographystyle{amsplain}
	
\end{document}